\documentclass[12pt, twoside, leqno]{article}
\usepackage{amsmath,amsthm}
\usepackage{amssymb}

\def\0{\mathbf 0}
\def\1{\mathbf 1}

\numberwithin{equation}{section}
\newcounter{thm}[section]
\numberwithin{thm}{section}

\newtheorem{theorem}[thm]{Theorem}
\newtheorem{lemma}[thm]{Lemma}

\newtheorem{corollary}[thm]{Corollary}

\newtheorem{definition}[thm]{Definition}

\begin{document}
\baselineskip=17pt

\title{Measures on Boolean algebras}

 \author{Thomas Jech \\
    e-mail: thomas.jech@gmail.com    }
\date{}

\maketitle

\renewcommand{\thefootnote}{}



\renewcommand{\thefootnote}{\arabic{footnote}}
\setcounter{footnote}{0}

\begin{abstract}
We present a necessary and sufficient condition for a Boolean algebra to carry a finitely additive measure.
\end{abstract}

\section{Introduction.}

Since the 1930's there has been a considerable interest in describing, ``algebraically'', those Boolean algebras that carry a (strictly positive) finitely additive measure. The problem is somewhat vague but the idea is that the necessary and sufficient condition should be, on its face, removed from measure
theoretic concepts, and if possible not too complicated. In this paper we introduce a property that
we believe is a step in that direction.

A \emph{Boolean algebra} is an algebra $B$ of subsets of a given
nonempty set $S$, with Boolean operations $a \cup b$, $a \cap b$,
$-a=S-a$, and the zero and unit elements $\mathbf 0=\emptyset$ and
$\mathbf 1=S$, and partial ordering $a\leq b$ iff $a\subset b$.

\begin{definition}\label{meas}
A \emph{measure}  on a Boolean
algebra $B$ is a real valued function $m$ on $B$ such
that
\begin{itemize}
\item [(i)] $m(\mathbf 0) = 0,\ m(a) > 0 \text{ for } a \not = \mathbf 0
        \text{, and } m(\mathbf 1) = 1, $
\item [(ii)] $m(a \cup b) =m(a) + m(b) \text{ if } a\cap b =\mathbf
0.$
\end{itemize}

A \emph{measure algebra} is an atomless Boolean algebra that
carries a measure.
\end{definition}

We let $B^+=B-\{\0\}$; an antichain is a set $A\subset B^+$ of pairwise disjoint elements.
Every measure algebra satisfies ccc, the \emph{countable chain condition}, i.e. every
antichain is countable.

In the early history, the attention was focused on $\sigma$-additive measures on Boolean 
$\sigma$-algebras. The first attempt at an algebraic description was by John
von Neumann who, in 1937, observed that if $B$ carries a $\sigma$-additive
measure then it satisfies the weak distributive law, in addition to the countable
chain condition. He then asked whether these two properties are sufficient for
measurability. See Problem 163 in the Scottish book \cite{Sc}.

A major advance toward an algebraic description of measure algebras was the work of Dorothy Maharam \cite{Mah} who, in 1947, introduced continuous submeasures (we now call such algebras Maharam algebras) and:

(1) gave a characterization of Maharam algebras in terms of the sequential topology,

(2) observed that a Suslin algebra (if it exists) is a counterexample to the von Neumann Problem, and

(3) asked whether a Maharam algebra must carry a measure.

The last question morphed into the famous ``Control Measure Problem'' in Functional Analysis.

As for finitely additive measures,  Alfred Tarski conjectured in late 1940's that it suffices that
$B^+$ is the union of countably many sets $C_n$ such that for each $n$, every antichain 
$A\subset C_n$ has at most $n$ elements, see \cite{HT}. 
This conjecture was refuted by Haim Gaifman in 1964, see \cite{Ga}.

In 1959, John Kelley reduced von Neumann's Problem to finitely additive measures and proved a 
necessary and sufficient condition for $B$ to carry a finitely additive measure, see \cite{Ke}:

(1) A Boolean $\sigma$-algebra $B$ carries a $\sigma$-additive measure if and only if
it is weakly distributive and carries a finitely additive measure.

(2) B is a measure algebra if and only if $B^+=\bigcup_n C_n$ such that each $C_n$
has a positive intersection number.

It turns out that Kelley rediscovered (1) which was proved (before 1950) by A. G. Pinsker,
see \cite{KVP}. As for Kelley's Intersection Number, we return to it in Section 4.

In 1980, Michel Talagrand \cite{Ta}, working on the Control Measure Problem,
introduced two properties of submeasures on Boolean  algebras, 
\emph{exhaustive} submeasures and \emph{uniformly exhaustive} submeasures.
In 2006 he solved the Control Measure Problem (see \cite{Tal}) by constructing an exhaustive
submeasure on a countable Boolean algebra that is not uniformly exhaustive.

In 1983, Nigel Kalton and James W. Roberts \cite{KR} introduced an ingenious 
combinatorial method and proved that if $B$ carries a uniformly exhaustive
submeasure then $B$ carries a measure. We shall employ their method in 
Section  5.

\section{$M$-ideals}

We employ a different approach. The key concept is that of an ideal on the set of all 
infinite sequences in $B^+$. We define an $M$-ideal and prove that the existence
of such an ideal is a necessary and sufficient condition for $B$ to be a measure
algebra.

\begin{definition}
Consider the set of all infinite sequences $s=\{a_n\}_n$ in $B^+$. A set
$I$ of such sequences is an \emph{$M$-ideal} if it has the following properties:

(M1) If $\{a_n\}_n \in I$ then $\bigwedge_n a_n = \0$, i.e. there is no $a>\0$ such that $a\leq a_n$ for all $n$.

(M2) If $s\in I$ and if $t$ is an infinite subsequence of $s$ then $t\in I$.

(M3) If $\{a_n\}_n \in I$ and $b_n\le a_n$ for all $n$ then $\{b_n\}_n \in I$.

(M4) If $\{a_n\}_n \in I$ and $\{b_n\}_n \in I$ then $\{a_n \cup b_n\}_n$ has an infinite 
subsequence that is in $I$.

(M5) If $\{a^k_n\}_n\in I$ for every $k$, then $\{a^n_n\}_n \in I$.

(M6) If $A_n$ are finite antichains in $B$ and $|A_n|\ge n$, then there exist
$a_n \in A_n$ such that $\{a_n\}_n \in I$.

\end{definition}

If $m$ is a measure on $B$, let $I$ be the set of all $\{a_n\}_n$ such that
$m(a_n)\le 1/n$. Then $I$ is an $M$-ideal.

We shall prove

\begin{theorem} If $B$ has an $M$-ideal then $B$ is a measure algebra.

\end{theorem}

Here is an outline of the proof: First we use the $M$-ideal to construct a
``fragmentation'' $\{C_n\}_n$ of $B$ with certain properties, one of them
witnessing the $\sigma$-bounded chain condition. Then we show that the
fragmentation is ``graded''. In Section 5 we use the Kalton-Roberts method
to show that each $C_n$ has a positive Kelley Intersection Number.
By Kelley, it follows that $B$ carries a measure.

\section{Fragmentations}

\begin{definition}

A \emph{fragmentation} of a Boolean algebra $B$ is a sequence of subsets
$C_1\subset C_2\subset...\subset C_n\subset ...$ such that
$\bigcup_n C_n =B^+$ and for every $n$, if $a\in C_n$ and $a\le b$ then $b\in C_n$.

A fragmentation is \emph{$\sigma$-bounded cc} if for every $n$
there is a constant $K_n$ such that every antichain $A\subset C_n$
has size $\le K_n$.

A fragmentation is \emph{graded} if for every $n$, whenever
$a\cup b\in C_n$ then either $a\in C_{n+1}$ or $b\in C_{n+1}$.

\end{definition}

Let $I$ be an $M$-ideal on $B$; we shall use $I$ to construct a 
graded $\sigma$-bounded cc fragmentation of $B$.

For each $n$ let

$$C_n=\{a\in B^+: a\ne a_n \text{ for all }  \{a_j\}_j \in I \}.$$

We show that  $\{C_n\}_n$ is a fragmentation:

First, if $a\in C_n$ and if $a\le b$ then $b\in C_n$: If not then $b=a_n$ for some
$\{a_n\}_n \in I$, and the sequence obtained from $\{a_n\}_n$ by replacing
$a_n$ by $a\le a_n$ is also in $I$, by (M3), and hence $a\notin C_n$.

Second, we show $C_n\subset C_{n+1}$: If $x \notin C_{n+1}$ then $x=a_{n+1}$
for some $\{a_k\}_k\in I$, and then the sequence $\{a_2, a_3,...,a_{n+1},...\}$,
in $I$ by (M2), witnesses that $x\notin  C_n$.

And third, if $a>\0$ is such that $a\notin C_k$ for all $k$, then there are 
sequences $\{a^k_n\}_n \in I$ such that $a=a_k^k$ for all $k$. But then the constant
sequence $\{a\}_n$ is in $I$ by (M5), contradicting (M1).

\begin {lemma}
Let $N$ be the set of all natural numbers.
For every $n\in N$ and every $\{a_k\}_k \in I$, the set $\{a_k: k\in N\}$ is not a subset of $C_n$.
\end{lemma}

\begin{proof}
If $\{a_k\}_k \in I$ then $a_n \notin C_n$.
\end{proof}

\begin{lemma}
The fragmentation is $\sigma$-bounded cc.
\end{lemma}

\begin{proof}
Let $n\in N$, and assume that $C_n$ has arbitrarily large finite antichains. For
each $k\in N$ let $A_k$ be an antichain in $C_n$ of size at least $k$. By (M6)
there is a sequence $\{a_k\}_k \in I$ such that $a_k\in A_k$ for all $k$.
Then $\{a_k: k\in N\}$ is a subset of $C_n$, a contradiction.
\end{proof}

\begin{lemma}
If $\{a_k\}_k$ is a sequence such that $a_k\notin C_k$ for every $k$, 
then $\{a_k\}_k\in I$.
\end{lemma}

\begin{proof} 
For every $k$ there exists a sequence 
$\{a^k_n\}_n \in I$ such that $a_k=a^k_k$. 
By (M5), $\{a^k_k\}_k \in I$.
\end{proof}

\begin{lemma}
For every $n$
there exists a $k>n$ such that for every $c\in C_n$, if $c=a\cup b$
then either $a\in C_k$ or $b\in C_k$.
\end{lemma}

\begin{proof}
Otherwise, for every $k$ there exist $c_k=a_k\cup b_k \in C_n$
such that $a_k\notin C_{n+k}$ and $b_k\notin C_{n+k}$. 
Clearly, $a_k\notin C_k$ and $b_k\notin C_k$. By Lemma 3.4.
$\{a_k\}_k\in I$ and $\{b_k\}_k\in I$, and by (M4),
$\{c_k\}_k=\{a_k \cup b_k\}_k$ has an infinite 
subsequence that is in $I$. That subsequence is included
in $C_n$, contrary to Lemma 3.2.
\end{proof}

Therefore $\{C_n\}_n$ has a subfragmentation that is graded, and we have

\begin{corollary}
If $B$ has an $M$-ideal then it has a graded $\sigma$-bounded cc
fragmentation.
\end{corollary}

In Section 5 we obtain a measure on $B$ under
the assumption that $B$ has a graded
$\sigma-$bounded cc fragmentation.

\section{Kelley's Theorem}

In this Section we introduce Kelley's condition for the existence
of finitely additive measure on a Boolean algebra.

Let $B$ be a Boolean set algebra, $B\subset P(S)$ for some set $S$.

\begin{definition}
Let $C$ be a subset of $B^+$. For every finite sequence $s=\langle
c_1,...c_n\rangle$ in $C$, let $\kappa_s=k/n$ where $k$ is the
largest size of a subset $J\subset\{1,...,n\}$ such that $\bigcap_{i\in J} c_i$
is nonempty. The \emph{intersection number} of $C$ is the infimum
$\kappa=\inf \kappa_s$ over all finite sequences $s$ in $C$.
\end{definition}

The sequences $s$ do not have to be nonrepeating.

Note that for any $n_0$, the infimum $\inf \kappa_s$ taken over all
sequences $s$ of length $n\ge n_0$ is still $\kappa$: if $s$ is a sequence
of length $n<n_0$, let $t$ be such that $t\cdot n\ge n_0$, and let $s^*$ be
a sequence we get when repeating each term of $s$ $t$-times. Then
$\kappa_{s^*}=\kappa_s$.

\begin{theorem}
(Kelley, \cite{Ke}.) Let $C\subset B^+$ have a positive intersection
number $\kappa$. Then there exists a finitely additive measure $m$ on $B$,
not necessarily strictly positive,
such that $m(c)\ge \kappa$ for all $c\in C$.
\end{theorem}

\begin{corollary}
If a Boolean algebra $B$ has a fragmentation $\{C_n\}$ such that
each $C_n$ has a positive intersection number, then $B$
carries a strictly positive finitely additive measure.
\end{corollary}

\section{The Kalton-Roberts Method}

We complete the proof by proving the following:

\begin{lemma}
Let $B$ be a Boolean algebra that has a graded $\sigma$-bounded
cc fragmentation $\{C_n\}$. Then for every $n$, $C_n$ has a positive intersection
number.
\end{lemma}

To prove the lemma, we adapt the Kalton-Roberts proof from \cite{KR}
that shows that a uniformly exhaustive submeasure is equivalent to  a measure.

The Kalton-Roberts proof uses the following combinatorial lemma
(Proposition 2.1 of \cite {KR}) which they called ``well known''. Part (a)
is verified by a counting argument; Part (b) follows from Part (a) by Hall's
``Marriage Theorem'' , see B\'ela Bollob\'as: ``Modern Graph Theory'' (1998), pp. 77-78.

\begin{lemma}
Let $M$ and $P$ be finite sets with $|M|=m$ and $|P|=p\le m$,
and let $k$, $3\le k \le p$ be an integer such that $p/k \ge 15\cdot m/p$.
Then

(a) There exists an indexed family $\{A_i: i\in M\}$ such that each $A_i$ is a three
point subset of $P$ and such that for every $I\subset M$ with $|I|\le k$,
$|\bigcup_{i\in I} A_i|>|I|.$

(b) It follows that for every $I\subset M$ with $|I|\le k$ there exists a one-to-one
choice function $f_I$ on $\{A_i: i\in I\}$.
\end{lemma}

We shall now apply the Kalton-Roberts method to prove Lemma 5.1.

{\it Proof of Lemma 5.1.}
Let $\{C_n\}$ be a graded $\sigma$-bounded cc fragmentation
of a Boolean algebra $B$, and let us fix an integer $n$. We prove that
the intersection number of  $C_n$ is positive, namely $\ge 1/(30K^2)$
where $K=K_{n+2}$ is the maximal size of an antichain in $C_{n+2}$.

We show that for every $m\ge 100K^2$, and every sequence
$\{c_1,...,c_m\}$ in $C_n$ there exists  some $J\subset m$
of size $\ge m/(30K^2)$ such that $\bigcap_{i\in J} c_i$ is nonempty.

Let $M=\{1,...,m\}$ with $m\ge 100K^2$ and let $c_1,...,c_m\in C_n$.
For each $I\subset M$, let

$$b_I = \bigcap\{c_i:i\in I\}\cap\bigcap\{-c_i:i\notin I\}.$$ The
sets $b_I$ are pairwise disjoint (some may be empty) and
$\bigcup\{b_I: I\subset M\}=\1$. Note that for each $i\in M$,
$\bigcup\{b_I: i\in I\}=c_i$. We shall find a sufficiently large
set $J\subset M$ with nonempty $b_J$.

We shall apply Lemma 5.2. First let $k\ge 3$ be the largest $k$ such that
$k/m < 1/(30K^2)$ (there is such because $3/m\le 3/(100K^2).$
We have $k<m$ and $(k+1)/m\ge 1/(30K^2)$. Then let $p$ be the largest $p\ge k$
such that $p/m<1/K$ (there is such because $k/m<1/K$.)

We verify the assumption of the lemma, $p/k\ge15m/p$
(using $p/(p+1)\ge 3/4$):
$$\frac{p}{k}=\frac{p}{p+1}\cdot \frac{p+1}{m}\cdot\frac{m}{k}\ge \frac{3}{4}\cdot \frac{1}{K}\cdot 30K^2\ge 20K$$
and
$$15\frac{m}{p}\le 15\cdot\frac{p+1}{p}\cdot \frac{m}{p+1}\le 15\cdot \frac{4}{3}\cdot K=20K.$$

Now we apply the Lemma: Let $P=\{1,...,p\}$. There exist three point sets $A_i\subset P$,
$i\in M$, and one-to-one functions $f_I$ on all $I\subset M$ of size $\le k$
with $f_I(i)\in A_i$ for all $i\in I$.

We shall prove that there exists a $J\subset M$ of size $\ge k+1$ (and hence $\ge m/(30K^2)$) such that
$b_J$ is nonempty. By contradiction, assume that there is no such $J$. Then
$$\bigcup\{b_I:|I|\le k\}=\1 \text{ and for each } i\in M, \, c_i=\bigcup\{b_I:|I|\le k \text { and } i\in I\}.$$
For each $i\in M$ and $j\in P$ let
$$a_{ij}=\bigcup\{b_I: |I|\le k,\, i\in I \text{ and } f_I(i)=j\}.$$
Note that for each $i\in M$, $c_i =a_{i,j_1}\cup a_{i,j_2}\cup a_{i,j_3}$ where $A_i=\{j_1,j_2,j_3\}$.

Let $j\in P$. We claim that the $a_{ij}$, $i\in M$, are pairwise disjoint:
If $a_{i_1,j}\cap a_{i_2,j}$ is nonempty, then because the $b_I$ are pairwise disjoint
there is some $I$ such that $i_1\in I$ and $i_2\in I$, and because
$f_I(i_1)=j=f_I(i_2)$ and $f_I$ is one-to-one, we have $i_1=i_2$.
Hence the $a_{ij}$, $i\in M$, are pairwise disjoint, and so only
at most $K$ of them belong to $C_{n+2}$.

Consequently, at most $p\cdot K$ of the $a_{ij}$ belong to
$C_{n+2}$ and because $pK<m$, there exists an $i$ such that
$a_{ij}\notin C_{n+2}$ for all (three) $j\in A_i$.

But then $c_i=a_{i,j_1}\cup a_{i,j_2}\cup a_{i,j_3}\notin C_n$ because
the fragmentation is graded. This contradicts the assumption that $c_i\in C_n$.


\section{Final Remarks}

1. Our proof does not construct a measure on $B$ outright but uses Kelley's Theorem, 
which employs the Hahn-Banach Theorem, known to require a version of the Axiom of Choice. 
This is to be expected, as the existence of measures on Boolean algebras is 
known to need the Hahn-Banach Theorem (W.A.J. Luxemburg).

 2. If $m$ is a measure on a Boolean algebra then the fragmentation defined by
 $C_n = \{ a\in B^+: m(a)\ge 1/{2^n}\}$ is graded and $\sigma$-bounded cc. But the
 same is true if $m$ is only a uniformly exhaustive submeasure. Thus the proof of
 Lemma 5.1 is also a proof of the Kalton-Roberts Theorem.
 
 3. If a Boolean $\sigma-$algebra $B$ carries a $\sigma-$additive 
 measure, or is just a Maharam algebra, then it is weakly distributive.
 
 $B$ is a Maharam algebra if and only if it is uniformly weakly distributive (Balcar-Jech)
 if and only if it is weakly distributive and $\sigma-$finite cc (Todorcevic).
 
Talagrand's construction yields a Maharam algebra that is not a measure algebra and is 
 $\sigma-$bounded cc.
 
 $B$ carries a $\sigma-$additive 
 measure if and only if it is weakly distributive and uniformly concentrated (Jech).
 
4. Condition (M5) cannot be relaxed. If there exists a Suslin tree and if $B$ is the
corresponding Suslin algebra then $B$ is not a measure algebra (it is not
$\sigma$-finite cc) but the ideal of all sequences converging to $\0$ satisfies 
(M1)-(M4) and (M6), as well as this weaker version of (M5): if for every $k$, 
$\lim_n a^k_n =\0$ then there exist $n_k$ such that $\lim_n a^k_{n_k} =\0$.

\bibliographystyle{plain}

\end{document}